\documentclass{amsart}
\usepackage{url}

\usepackage{verbatim}
\usepackage{amssymb}
\usepackage{upref}
\usepackage[all]{xy}
\usepackage{color}

\allowdisplaybreaks

\newtheorem{thm}{Theorem}[section]

\newtheorem{cor}[thm]{Corollary}

\theoremstyle{definition}

\newtheorem{q}[thm]{Question}

\theoremstyle{remark}
\newtheorem{rem}[thm]{Remark}

\numberwithin{equation}{section}

\newcommand{\thmref}[1]{Theorem~\textup{\ref{#1}}}

\newcommand{\qref}[1]{Question~\textup{\ref{#1}}}

\newcommand{\midtext}[1]{\quad\text{#1}\quad}

\renewcommand{\)}{\textup)}
\newcommand{\ie}{\emph{i.e.}}

\newcommand{\DD}{\mathcal D}

\newcommand{\BB}{\mathcal B}

\renewcommand{\AA}{\mathcal A}

\renewcommand{\epsilon}{\varepsilon}

\renewcommand{\>}{\rangle}

\renewcommand{\bar}{\overline}

\newcommand{\ME}{\underset{\text{M}}{\sim}}

\begin{document}
\title[Fell bundles and imprimitivity theorems]{Fell bundles and imprimitivity theorems: Mansfield's and Fell's theorems}

\author[Kaliszewski]{S. Kaliszewski}
\address{Department of Mathematics and Statistics
\\Arizona State University
\\Tempe, Arizona 85287}
\email{kaliszewski@asu.edu}

\author[Muhly]{Paul S. Muhly}
\address{Department of Mathematics
\\The University of Iowa
\\Iowa City, IA 52242}
\email{paul-muhly@uiowa.edu}

\author[Quigg]{John Quigg}
\address{Department of Mathematics and Statistics
\\Arizona State University
\\Tempe, Arizona 85287}
\email{quigg@asu.edu}

\author[Williams]{Dana P. Williams}
\address{Department of Mathematics
\\Dartmouth College
\\Hanover, NH 03755}
\email{dana.williams@dartmouth.edu}



\date{July 24, 2012}

\begin{abstract}
In the third and latest paper in this series, we recover the imprimitivity theorems of Mansfield and Fell using our technique of Fell bundles over groupoids.
Also, we apply the Rieffel Surjection of the first paper in the series
to relate our version of Mansfield's theorem to that of an Huef and Raeburn, and to give an automatic amenability result for certain transformation Fell bundles.
\end{abstract}
\maketitle

\section{Introduction}
\label{intro}

This is a sequel to our earlier papers \cite{kmqw2, kmqw3},
and completes our task of showing that all known imprimitivity theorems involving groups can be unified via the Yamagami-Muhly-Williams equivalence theorem (to which we will refer as the YMW Theorem) \cite{mw:fell, yam:symmetric}, which shows how an equivalence between Fell bundles gives rise to a Morita equivalence between their $C^*$-algebras.
In \cite{kmqw2} we showed how the YMW Theorem can be used to derive Raeburn's symmetric imprimitivity theorem (which, as Raeburn points out, quickly implies both the Green-Takesaki imprimitivity theorem for induced representations of $C^*$-dynamical systems and Green's imprimitivity theorem for induced actions). To this end, we first proved what we called the Symmetric Action Theorem for commuting free and proper actions by automorphisms of groups on Fell bundles over groupoids. In \cite{kmqw2} we also proved what we called the One-Sided Action Theorem, a special case of the Symmetric Action Theorem with one group trivial. We also proved a structure theorem characterizing free and proper actions on Fell bundles: using a result perhaps due to Palais, such actions all arise from transformation Fell bundles, which were studied in \cite{kmqw1}.

In \cite{kmqw3} we studied the One-Sided Action Theorem further, deriving a curious connection with Rieffel's imprimitivity theorem for generalized fixed-point-algebras: the imprimitivity bimodule in Rieffel's theorem is a quotient of the imprimitivity bimodule in the One-Sided Action Theorem.
Consequently, it is reasonable to regard the $C^*$-algebra of an orbit Fell bundle by a free and proper group action as a ``universal'', or ``full'', version of a Rieffel-type generalized fixed-point algebra.

In the current paper we show how the YMW Theorem can be used to prove both Mansfield's imprimitivity theorem, which is the dual to the Green-Takesaki theorem, and Fell's original imprimitivity theorem, which involves the restriction of a Fell bundle to a subgroup.

In addition, we apply the Rieffel Surjection of \cite{kmqw3} to relate our version 
of Mansfield's theorem to that of an Huef and Raeburn in \cite{aHRMansfield},
and we further give an automatic amenability result for transformation Fell bundles of the form $\BB\times G/H$,
where $\BB\to G$ is a Fell bundle over a group and
 $H$ is an amenable subgroup of $G$.

\section{Preliminaries}\label{prelims}

We adopt the conventions of \cite{kmqw1,kmqw2}.
All our Banach bundles will be upper semicontinuous and separable, all our spaces and groupoids will be locally compact Hausdorff and second countable, and our groupoids will all have left Haar systems.
Convenient references for the various types of coactions (reduced, full, normal, maximal) we discuss are \cite{boiler} and \cite[Introduction]{kmqw1}.

In order to place our version of Mansfield's theorem in context, it is perhaps helpful to include a short history of this imprimitivity theorem.
Mansfield's original imprimitivity theorem \cite[Theorem~27]{mansfield} states that, if $\delta$ is a reduced coaction of a locally compact group $G$ on a $C^*$-algebra $A$, and $H$ is a closed amenable normal subgroup of $G$, then $\delta$ restricts to a reduced coaction $\delta|$ of the quotient group $G/H$ on $A$, and there is a Morita equivalence
\begin{equation}\label{mansfield reduced}
A\rtimes_\delta G\rtimes_{\hat\delta|,r} H\ME A\rtimes_{\delta|} G/H.
\end{equation}
Switching from reduced to full coactions, the amenability hypothesis was removed in \cite[Theorem~3.3]{KalQuiMansfield}, where \eqref{mansfield reduced} was proved under the assumption that the coaction $\delta$ is normal.

On the other hand, if $\delta$ is maximal (and $H$ is any closed normal subgroup of $G$), \cite[Theorem~5.3]{KalQuiFullMansfield} gives a version of the Mansfield imprimitivity theorem for the full crossed product by the dual action:
\begin{equation}\label{mansfield full}
A\rtimes_\delta G\rtimes_{\hat\delta|} H\ME A\rtimes_{\delta|} G/H.
\end{equation}
Here the restricted coaction $\delta|$ of $G/H$ is also maximal, by \cite[Corollary~7.2]{KalQuiFullMansfield}.

Theorem~3.1 of
\cite{eq:full} says that if $p:\BB\to G$ is a Fell bundle over a discrete group $G$, and if $H$ is any subgroup of $G$, then
\begin{equation}\label{EQ}
C^*(\BB)\rtimes_\delta G\rtimes_{\hat\delta|} H\ME C^*(\BB\times G/H),
\end{equation}
where 
$\delta$ is the canonical coaction of $G$ on $C^*(\BB)$, determined by $\delta(b)=b\otimes p(b)$ for $b\in\BB$,
and
$\BB\times G/H\to G\times G/H$ is the transformation Fell bundle (as in \cite[Section~4]{kmqw1}) associated to the action of $G$ on itself by left translation.
When the subgroup $H$ is normal, \eqref{EQ} is a special case of \eqref{mansfield full},
because 
by \cite[Proposition~4.2]{ekq} the 
coaction $\delta$ of $G$ on $C^*(\BB)$ is maximal, and
by \cite[Corollary~2.12]{eq:full} we have $C^*(\BB\times G/H)\cong C^*(\BB)\rtimes_{\delta|} G/H$.

Back to reduced coactions, but removing the hypothesis of normality (as well as the amenability) of $H$, \cite[Theorem~5.1]{EKRHomogeneous} and \cite[Theorem~3.1]{aHRMansfield} give a version of Mansfield's imprimitivity theorem for homogeneous spaces:
\begin{equation}\label{homogeneous reduced}
A\rtimes_\delta G\rtimes_{\hat\delta|,r} H\ME A\rtimes_{\delta,r} G/H,
\end{equation}
where now $A\rtimes_{\delta,r} G/H$ is defined as the closed span of the products $j_A(a)\bar{j_G}(f)$ for $a\in A$ and $f\in C_0(G/H)$, and where the latter is identified with its canonical image in $C_b(G)=M(C_0(G))$.

It is natural to ask:
\begin{q}\label{mansfield question}
Is there a common generalization of \eqref{mansfield full} and \eqref{homogeneous reduced}?
\end{q}
Such a generalization would be a version of \eqref{mansfield full} for arbitrary closed subgroups $H$, and also a version of \eqref{homogeneous reduced} for full crossed products.
More precisely, such a result would (presumably) say that if $(A,\delta)$ is a maximal coaction of a locally compact group $G$ and $H$ is a closed subgroup of $G$, then the full crossed product $A\rtimes_\delta G\rtimes_{\hat\delta} H$ is Morita equivalent to a ``restricted crossed product'' $A\rtimes_{\delta|} G/H$.
However, it is not clear how to get an appropriate analogue of the restricted crossed product $A\rtimes_{\delta|} G/H$.
\thmref{mansfield new} below will give a version of such a result in the case that $A=C^*(\BB)$ for a Fell bundle $\BB\to G$. 
This will not completely answer \qref{mansfield question}, because, while it is true that every maximal coaction is Morita equivalent to one of the form $(C^*(\BB),\delta)$, 
the restricted crossed product is usually identified with a subalgebra of the multiplier algebra $M(A\rtimes_\delta G)$, and 
there is no mechanism for inducing arbitrary $C^*$-subalgebras across imprimitivity bimodules.

\section{Mansfield's imprimitivity theorem}
\label{mansfield}

\begin{thm}\label{mansfield new}
Let $\BB\to G$ be a Fell bundle over a locally compact group, and let $H$ be a closed subgroup of $G$. Let $\delta$ be the canonical coaction of $G$ on $C^*(\BB)$, and let 
$\hat\delta|$ be the restriction to $H$ of the dual action of $G$ on $C^*(\BB)\rtimes_\delta G$.
Further let $\BB\times G/H\to G\times G/H$ denote the transformation Fell bundle associated to the action of $G$ on $G/H$ by left translation.
Then there is a Morita equivalence
\begin{equation}\label{mansfield new equation}
C^*(\BB)\rtimes_\delta G\rtimes_{\hat\delta|} H\ME C^*(\BB\times G/H).
\end{equation}
\end{thm}

\begin{proof}
By translation in the second coordinate, $H$ acts by automorphisms on the right of the transformation Fell bundle $p:\BB\times G\to G\times G$.
Applying the One-Sided Action theorem \cite[Corollary~2.3]{kmqw2}
gives the Fell-bundle equivalence
\begin{equation}\label{equivalence}
(\BB\times G)/H\ME H\ltimes (\BB\times G).
\end{equation}
Using the obvious isomorphism
\[
(\BB\times G)/H\cong \BB\times G/H,
\]
apply the Yamagami-Muhly-Williams equivalence theorem
\cite[Theorem~6.4]{mw:fell}
to \eqref{equivalence} to get a Morita equivalence
\[
C^*(\BB\times G/H)\ME C^*(H\ltimes (\BB\times G))
\]
Switching the sides and applying the isomorphism 
$C^*(H\ltimes (\BB\times G))\cong C^*(\BB\times G)\rtimes H$
from \cite[Theorem~7.1]{kmqw1}, where the action of $H$ on $C^*(\BB\times G)$ 
is associated to right-translation in the second coordinate on the Fell bundle $\BB\times G$,
gives
\[
C^*(\BB\times G)\rtimes H\ME C^*(\BB\times G/H).
\]
Finally, \cite[Theorem~5.1]{kmqw1}
gives an isomorphism $C^*(\BB\times G)\cong C^*(\BB)\rtimes_\delta G$,
which, by the proof of \cite[Proposition~8.2]{kmqw1}, carries 
the action of $H$ on $C^*(\BB\times G)$ to
the (restriction to $H$ of the) dual action on the crossed product $C^*(\BB)\rtimes_\delta G$,
and the result follows.
\end{proof}

In \thmref{mansfield new}, the right-hand $C^*$-algebra in the Morita-equivalent pair is the Fell-bundle algebra $C^*(\BB\times G/H)$. 
However, as we have seen in the introduction, in most versions of Mansfield imprimitivity this algebra is 
some sort of crossed product of $C^*(\BB)$ by a restriction, $\delta |$, of $\delta$ to $G/H$.
In the case of \thmref{mansfield new},
by analogy with the notation in \eqref{homogeneous reduced}, it seems reasonable to regard $C^*(\BB\times G/H)$ as a ``full crossed product'' $C^*(\BB)\rtimes_{\delta|} G/H$ by the (heretofore undefined) restricted coaction $\delta|$ of the homogeneous space $G/H$.
On the other hand, when $H$ is normal it seems prudent to check whether $C^*(\BB\times G/H)$ is isomorphic to the crossed product $C^*(\BB)\rtimes_{\delta|} G/H$ by the (well-defined) restricted coaction $\delta|$ of the quotient group $G/H$ on $C^*(\BB)$. Fortunately, this is indeed the case:

\begin{thm}\label{normal}
Let $\BB\to G$ be a Fell bundle over a locally compact group, let $H$ be a closed normal subgroup of $G$, and let $\delta$ be the canonical coaction of $G$ on $C^*(\BB)$. Then there is an isomorphism
\[
\theta:C^*(\BB)\rtimes_{\delta|} G/H\to C^*(\BB\times G/H)
\]
such that
\begin{equation}\label{theta H}
\theta\bigl(j_{C^*(\BB)}(f)j_G(g)\bigr)
=(\Delta^{1/2}f)\boxtimes g
\midtext{for}f\in \Gamma_c(\BB),g\in C_c(G),
\end{equation}
where $\Delta$ is the modular function of $G$ and $(\Delta^{1/2}f)\boxtimes g$ denotes the 
element of $\Gamma_c(\BB\times G/H)$ defined by
\[
(f\boxtimes g)(s,tH)=(\Delta(s)^{1/2}f(s)g(tH),tH).
\]
\end{thm}

\begin{proof}\
\cite[Proposition~2.1]{kmqw3} gives us nondegenerate homomorphisms $\Phi$ and $\mu$ of $C^*(\BB)$ and $C_0(G/H)$, respectively, into $M(C^*(\BB\times G/H))$; we need to know that the pair $(\Phi,\mu)$ is covariant. It will then follow from \cite[Proposition~2.1]{kmqw3} that the integrated form $\theta:=\Phi\times\mu$ is surjective, and it will remain to show that $\theta$ is injective. Luckily, the hard work has already been done: the covariance and the injectivity can be proven via routine adaptations of the proof of \cite[Theorem~5.1]{kmqw1} (with the proof of covariance using a suitable routine adaptation of \cite[Proposition~3.4]{kmqw1}).
\end{proof}

\section{Mansfield and the Rieffel Surjection}

Let $\BB\to G$ be a Fell bundle over a locally compact group, and let $H$ be a closed subgroup of $G$.
Let $X$ be the $C^*(\BB)\rtimes_\delta G\rtimes_{\hat\delta|} H-C^*(\BB\times G/H)$ imprimitivity bimodule from \thmref{mansfield new}.
Then $H$ acts freely and properly on (the right of) the Fell bundle $\BB\times G\to G\times G$,
and the orbit Fell bundle is isomorphic to $\BB\times G/H\to G\times G/H$,
so by \cite[Theorem~3.1]{kmqw3} we have a Rieffel Surjection
\begin{multline}\label{Rie sur}
(\Lambda,\Upsilon,\Phi):
(C^*(\BB\times G)\rtimes_\alpha H,X,C^*(\BB\times G/H))
\\
\to
(C^*(\BB\times G)\rtimes_{\alpha,r} H,X_R,C^*(\BB\times G)^\alpha)
\end{multline}
of imprimitivity bimodules,
where $C^*(\BB\times G)^\alpha$ denotes the generalized fixed-point algebra.
We can replace the left-hand coefficient $C^*$-algebra $C^*(\BB\times G)\rtimes_\alpha H$ of $X$
by either of the isomorphic algebras
\[
C^*(\BB)\rtimes_\delta G\rtimes_{\hat\delta|} H
\midtext{or}
C^*((\BB\times G)\rtimes H).
\]
Similarly, for $X_R$ we can replace the left-hand coefficient $C^*$-algebra $C^*(\BB\times G)\rtimes_{\alpha,r} H$ 
by either of the isomorphic algebras 
$C^*(\BB)\rtimes_\delta G\rtimes_{\hat\delta|,r} H$
or, by \cite[Example~11]{SimsWilliamsReduced},
$C^*_r((\BB\times G)\rtimes H)$,
and the right-hand coefficient algebra $C^*(\BB\times G)^\alpha$ 
by either of the isomorphic algebras
$(C^*(\BB)\rtimes_\delta G)^{\hat\delta|}$
or, by \cite[Corollary~3.5]{kmqw3},
$C^*_r(\BB\times G/H)$.

In particular, we can write the Rieffel Surjection \eqref{Rie sur} as
\begin{multline}\label{mansfield surjection}
(\Lambda,\Upsilon,\Lambda):
(C^*(\BB\times G)\rtimes_\alpha H,X,C^*(\BB\times G/H)
\\
\to
(C^*(\BB\times G)\rtimes_{\alpha,r} H,X_R,C^*_r(\BB\times G/H)).
\end{multline}
When $H=\{e\}$, the following corollary generalizes \cite[Remark~2.11]{eq:full} from the discrete case, and is  unsurprising, since for group coactions the regular representation of the crossed product is faithful.
We should also mention that the following corollary follows from \cite[Theorem~1]{SimsWilliamsAmenable}, which is proved by different means,
since the transformation groupoid $G\times G/H$ is amenable in the sense of \cite{AR}, 
being groupoid-equivalent to the amenable group $H$. 

\begin{cor}\label{amenable}
Let $\BB\to G$ be a Fell bundle over a locally compact group, and let $H$ be a closed subgroup of $G$.
If $H$ is amenable,
then the transformation bundle $\BB\times G/H\to G\times G/H$ is metrically amenable in the sense of \cite{SimsWilliamsAmenable}, \ie, the regular representation
\[
\Lambda:C^*(\BB\times G/H)\to C^*_r(\BB\times G/H)
\]
is an isomorphism.
\end{cor}

\begin{proof}
This follows from \eqref{mansfield surjection}, because the first regular representation $\Lambda:C^*(\BB\times G)\rtimes_\alpha H\to C^*(\BB\times G)\rtimes_{\alpha,r} H$ is an isomorphism.
\end{proof}

\begin{rem}\label{compare aHR}
in \cite[Theorem~3.1]{aHRMansfield}, an Huef and Raeburn give a Morita equivalence
\begin{equation}\label{aHR}
C^*(\BB)\rtimes_\delta G\rtimes_{\hat\delta|,r} H
\ME
C^*(\BB)\rtimes_{\delta,r} G/H
\end{equation}
with an imprimitivity bimodule $\bar\DD$ that is a completion of Mansfield's algebra $\DD$ \cite{mansfield}.
They define $C^*(\BB)\rtimes_{\delta,r} G/H$ as the $C^*$-subalgebra
of $C^*(\BB)\rtimes_\delta G$ generated by $j_C^*(\BB)(C^*(\BB))\bar{j_G}(C_0(G/H)$, and they show that this coincides with Rieffel's generalized fixed-point algebra $(C^*(\BB)\rtimes_\delta G)^{\hat\delta|}$ associated to the action $\hat\delta|$ of $H$.
If follows from \cite[Lemma~3.2]{aHRMansfield} (see also \cite[Remark~3.4]{aHRMansfield}) that the imprimitivity bimodules $X_R$ and $\bar\DD$ are isomorphic.
Thus,  the an Huef-Raeburn Morita equivalence \eqref{aHR} is a quotient of that in 
\thmref{mansfield new}.
\end{rem}

\section{Fell's original imprimitivity theorem}
\label{fell section}

Finally, we derive one more well-known imprimitivity theorem from the YMW theorem, namely Fell's original imprimitivity theorem for $C^*$-algebraic bundles (\ie, Fell bundles) over groups. This one seems not to follow from the Symmetric Action theorem. 

To apply YMW theorem \cite[Theorem~6.4]{mw:fell}, we first need
a Fell-bundle equivalence:

\begin{thm}\label{fell}
Let $\AA\to G$ be a Fell bundle over a locally compact group, and let $H$ be a closed subgroup of $G$.
Let  $\AA\times G/H\to G\times G/H$ be the transformation Fell bundle \(where $G$ acts on $G/H$ by left translation\).
Let $\AA|_H\to H$ be the restricted Fell bundle.
Then $\AA$ gives an $(\AA\times G/H)-\AA|_H$ equivalence in the following way:
\begin{enumerate}
\item
$\AA\times G/H$ acts on the left of $\AA$ by
\[
(a,p(b)H)b=ab;
\]

\item
the left inner product is given by
\[
{}_L\<a,b\>=(ab^*,p(b)H);
\]

\item
$\AA|_H$ acts on the right of $\AA$ by right multiplication;

\item
the right inner product is given by
\[
\<a,b\>_R=a^*b.
\]
\end{enumerate}
\end{thm}

\begin{proof}
The computations required to verify the conditions of \cite[Definition~6.1]{mw:fell} are routine.
\end{proof}

We recover Fell's imprimitivity theorem \cite[Theorem~XI.14.17]{fd2}, which can be rephrased as follows:

\begin{cor}\label{fell original}
With the hypotheses of \thmref{fell}, 
$\Gamma_c(\AA)$ completes to a
$C^*(\AA\times G/H)-C^*(\AA|_H)$ imprimitivity bimodule.
\end{cor}

\begin{proof}
This follows immediately from \thmref{fell} and the YMW Theorem.
\end{proof}

\begin{rem}
The above proof of Fell's theorem is quite a bit shorter, and we believe more natural, than Fell and Doran's. Fell and Doran had to work quite hard, developing a version of the transformation bundle over $G$ that incorporates the left action of $G$ on $G/H$. Our job is much easier because we allow ourselves to consider the transformation Fell bundle $\AA\times G/H$ over the groupoid $G\times G/H$; Fell and Doran did not avail themselves of the technology of groupoids, so all their bundles had to be over groups.
\end{rem}


\providecommand{\bysame}{\leavevmode\hbox to3em{\hrulefill}\thinspace}
\providecommand{\MR}{\relax\ifhmode\unskip\space\fi MR }
\providecommand{\MRhref}[2]{%
  \href{http://www.ams.org/mathscinet-getitem?mr=#1}{#2}
}
\providecommand{\href}[2]{#2}

\end{document}